\documentclass[10pt]{amsart}
\usepackage{amssymb,amsmath}

\usepackage{graphicx}

\begin{document}

\newtheorem{theorem}{Theorem}[section]
\newtheorem{cor}[theorem]{Corollary}
\newtheorem{rmk}[theorem]{Remark}
\newtheorem{lemma}[theorem]{Lemma}
\newtheorem{gen}{Generalization}
\newtheorem{prop}[theorem]{Proposition}
\newtheorem{claim}[theorem]{Claim}
\newtheorem{observation}[theorem]{Observation}
\newtheorem{notation}[theorem]{Notation}
\newtheorem{conjecture}[theorem]{Conjecture}
\newtheorem{defin}[theorem]{Definition}
\newtheorem{defins}[theorem]{Definitions}

\newcommand{\map}{\mbox{$\rightarrow$}}
\newcommand{\bbb}{\mbox{$\beta$}}
\newcommand{\la}{\mbox{$\lambda$}}
\newcommand{\aaa}{\mbox{$\alpha$}}
\newcommand{\eee}{\mbox{$\epsilon$}}
\newcommand{\Rrr}{\mbox{$\mathbb{R}$}}
\newcommand{\lpd}{\mbox{$L^{\mathcal{V}(P,D^*)}$}}
\newcommand{\fpd}{\mbox{$\mathcal{V}(P,D^*)$}}
\newcommand{\bdd}{\mbox{$\partial$}}

\newcommand{\Li}{\mbox{$L_+^{in}$}}
\newcommand{\Lo}{\mbox{$L_+^{out}$}}

\title{Cut-disks for level spheres in link and tangle complements}
\author{Maggy Tomova}
\thanks{Research partially supported by an NSF grant.}
\maketitle

\begin{abstract}

In \cite{Wu} Wu shows that if a link or a knot $L$ in $S^3$ in thin position has thin spheres, then the thin sphere of lowest width is an essential surface in the link complement. In this paper we show that if we further assume that $L \subset S^3$ is prime, then the thin sphere of lowest width also does not have any vertical cut-disks. We also prove the result for a specific kind of tangles in $S^2 \times [-1,1]$.

\end{abstract}
\section{Introduction}

Thin position is a knot invariant first introduced by Gabai \cite{Ga} to prove property $R$. Since then it has become of significant interest in its own right. The idea is to consider a knot in $S^3$ in Morse position with respect to the standard height function and then to isotope the knot so that  if we consider a maximal collection of level spheres no two of which are parallel in the knot complement, the sum of their intersections with the knot is minimized. A level sphere is called thin if it intersects the knot in fewer points than the two level spheres adjacent to it do. One of the interesting problems is establishing the compressibility properties of the level spheres in the knot complement. The first result in this direction was given by Thompson. In \cite{Thom} she showed that if the knot is in thin position and has some thin spheres, its complement contains an incompressible, non-boundary parallel planar surface. The result was refined by Wu \cite{Wu} who identified the thin level sphere of minimum width to be such an incompressible surface. In \cite {Tom} the current author gave further restrictions on the possible compressing disks for the thin spheres. 

 Recently, it has proven useful to consider not only compressing disks for surfaces in the knot complement but also cut-disks. These are annuli in the knot complement that have one boundary component lying in the surface considered and the other in the boundary of a regular neighborhood of the knot so that the interior of the annulus is disjoint from the surface and it is not parallel in the knot complement to an annulus on the surface. The precise definition is given later. Surfaces that do not have cut-disks can be very useful; for example if $F$ is a meridional surface with no cut-disks and no compressing disks and $B$ is any other meridional surface in the knot complement, then there is an isotopy after which all curves of $F \cap B$ are essential in $B$. 
 
 As one may expect, there are fewer obstructions to a level sphere having cut-disks than having compressing disks. One reason for that is that a compressing disk allows us to perform isotopies in one of the two balls bounded by the union of the compressing disk and the level sphere without affecting the other ball. When we are considering a cut-disks such isotopies generally create new critical points in the arc of the knot that pierces the cut-disk thus affecting the width of the knot in ways that are difficult to track. In this paper we have resolved this problem with the additional hypothesis that the cut-disk is vertical with respect to the height function, i.e., the cut-disk can be isotoped in the knot complement to be the union of a vertical annulus, possibly once-punctured by the knot, and a possibly once-punctured level disk. We will show that the thin sphere of lowest width does not have any vertical cut-disks. In addition we extend some of the previously known results about compressing disks for level sphere for a knot in thin position to tangles in thin position.  
 
 \section{preliminaries}
 
In this paper a {\em link} is a closed 1-manifold with one or more components that is embedded in $S^3$. We assume that links cannot be split, i.e., there is no $2$-sphere in $S^3$ that separates the components of $L$. Let $F$ be a meridional surface embedded in the complement of a link $L$. A {\em cut-disk} for $F$ is a disk $D^c \subset S^3$ such that $D^c \cap F =\bdd D^c$, $|D^c \cap L|=1$ and the annulus $D^c-nbhd(L)$ is not parallel  in the link complement to a subset of $F-nbhd(L)$. In particular if $L$ is prime $\bdd D^c$ is not parallel to a boundary component of $F-nbhd(L)$. We use the term {\em c-disk} to refer to either a compressing or a cut-disk. A c-disk is vertical with respect to a height function if it is the union of a vertical annulus and a level disk either one of which, but not both, may be punctured by the knot once.

Consider a height function $\pi: S^3 \to \mathbb{R}$ such that
$\pi$ restricts to a Morse function on $L$. If $t$ is a regular
value of $\pi|_L$, $\pi^{-1}(t)$ is called a level sphere with width
$w(\pi^{-1}(t))=|L\cap \pi^{-1}(t)|$. If $c_{0}<c_{1}<...<c_{n}$ are all the
critical values of $\pi|_L$, choose regular values
$r_{1},r_{2},...,r_{n}$ such that $c_{i-1}<r_{i}<c_{i}$. Then the
{\em width of $L$ with respect to $\pi$} is defined by $w(L,\pi)=\sum
w(\pi^{-1}(r_{i}))$. The {\em width} of $L$, $w(L)$ is the minimum of $w(L,\pi')$
over all possible height functions $\pi'$. We say that $L$ is in thin position whenever we consider $L$ with respect to a height function $\pi$
which realizes its width, i.e., if $w(L,\pi)=w(L)$.

Instead of considering all possible height functions of $S^3$ we can consider all possible isotopic images of $L$. The height function $\pi$ suggests an isotopy $f$ of $S^3$ such that $w(f(L), h)=w(L,\pi)$ where $h$ is the standard height function on $S^3$, namely the projection onto the third coordinate. We will take this point of view for the rest of the paper and will assume that $\pi$ is the standard height function on $S^3$.  More details about thin position and basic results can be found in \cite{Schar1}.

A level sphere $\pi^{-1}(t)$ is called {\em thin}
if the highest critical point for $L$ below it is a maximum and the
lowest critical point above it is a minimum. If the highest critical point for $L$ below $\pi^{-1}(t)$ is a minimum and the
lowest critical point above it is a maximum the level sphere is called {\em thick}. As the lowest critical point of $L$ is a minimum and the highest is a maximum, a thick level sphere can always be found. It is possible that the link does not have any thin spheres with respects to some height function and in this case the link is said to be in {\em bridge position}.

Clearly there are multiple isotopic images of a link all giving the same width. We will need to take advantage of these so we give a precise definition of such isotopies.

\begin{defin}
An isotopy that, in the end, leaves $\pi|_L$ unchanged is called level preserving. An isotopy that leaves $\pi|_L$ unchanged throughout the isotopy will be called a horizontal isotopy. 
\end{defin}

Since $L$ is in general position with respect to $\pi$ it is
disjoint from both the minimum (south pole) and maximum (north
pole) of $\pi$ on $S^{3}$.  Thus the
width of $L$ could just as easily be computed via its
diffeomorphic image in $S^{2} \times \Rrr$.  As $L$ is compact we may in fact assume that $L$ is contained in  $S^{2} \times [-1,1]$. Finally, by general position, $L$ is disjoint from some fiber of $S^{2} \times [-1,1]$ so we can also regard it as contained in the ball $D^{2} \times [-1,1]$.  We will use all of these points of view and  continue to use $\pi$ to denote the
projection onto the third coordinate. When we consider $L$ to lie in $D^{2} \times [-1,1]$, the preimages $\pi^{-1}(t)$ are level disks which can be thin, thick or neither depending on the corresponding level spheres. Similarly the width of a level disk will be the width of the corresponding level sphere and a disk in $D^{2} \times [-1,1]$ is a c-disk for a level disk if it is a c-disk for the corresponding level sphere. 

A tangle in $S^2 \times [-1,1]$ is a properly embedded one manifold with possibly multiple components. We compute the width of a tangle analogously to the width of a link. The two level sphere $w(S^2 \times
\{-1\})$ and $w(S^2 \times \{1\})$ will be considered to be thin if the highest critical point of $L$ is a maximum and the lowest one is a minimum. When we consider the tangle to lie in $D^2 \times [-1,1]$ as we did with links, it is important to perform all isotopies keeping the endpoints in $(D^2 \times \{-1\}) \cup (D^2 \times \{1\})$ so the isotopies make sense as isotopies of $S^{2} \times [-1,1]$. Unless otherwise specified we will use $L$ to denote either a link or a tangle that we consider as a tangle in $S^2 \times [-1,1]$. We will use $T$ for tangles that are embedded in a ball and thus isotopies moving their endpoints on the boundary sphere will be allowed.

\section{Notation and some definitions} \label{sec:notationanddef}

In this paper $L$ will denote an unsplit link (possibly with only one component) or an unsplit tangle embedded in $S^2 \times[-1,1]$, $P$ will denote a level sphere for $L$ and $D^*$ will be a c-disk for $P$. Without loss of generality we will assume that $P=\pi^{-1}(0)$ and $D^*$ lies above $P$. Consider a closed regular neighborhood of $D^*$, $D^*_{\delta}=D^* \times [-\delta,\delta]$ that is sufficiently small so that  if $D^*$ is a compressing disk $|D^*_{\delta} \cap L|=0$ and if $D^*$ is a cut-disk $|D^*_{\delta} \cap L|=1$ and $L\cap D^*_{\delta}$ has no critical points. We will let $B^{in}$ be the ball cobounded by $P$ and, say, $D^2 \times \{-\delta\}$ in $S^2 \times [-1,1]$, and $B^{out}$ be the  3-manifold cobounded by $P$ and $D^2 \times \{\delta\}$ (if we consider $L$ as lying in $S^3$ with all endpoints meeting at infinity, $B^{out}$ is also a ball). We will refer to $B^{in}$ and $B^{out}$ as the {\em inside} and the {\em outside} of $D^*$ respectively. As $B^{in}$ and $B^{out}$ are separated by $D^*_{\delta}$ any isotopy of $B^{in}$ can be extended to the identity on $B^{out}$ and vice versa. If $L$ is a link we can choose either side to be the inside and we will always make this choice so that the level sphere lying directly above the highest maximum of $L\cap B^{in}$ intersects $L\cap B^{out}$. If $L$ is a tangle this condition is always satisfied.

The components of $L_+=L \cap \pi^{-1}[0,1]$ can be classified based on how they lie with respect to $D^*$. We will let $\Lo=L_+\cap B^{out}$ and $\Li=L_+\cap B^{in}$. If $D^*$ is a cut-disk, we will let $\tau$ be the component of $L_+$ that intersects $D^*$ ($\tau$ is not disjoint from $\Lo$ and $\Li$) and will call $\tau$ {\em the connecting strand}. We will let $\tau'=\tau \cap D^*_{\delta}$, by definition $\tau'$ does not contain any critical points. There are two possibilities to consider, the point $\Lo \cap \tau'$ is either higher or lower than the point $\Li \cap \tau'$. If the point $\Lo \cap \tau'$ is higher than the point $\Li \cap \tau'$ let $\aaa=\Lo$ and $\bbb = \Li$. In the other case reverse the labels $\aaa$ and $\bbb$. Thus our labeling guarantees that the point $\aaa \cap \tau'$ is  always higher than the point $\bbb \cap \tau'$ and we will say that $\tau$ is descending from $\aaa$ to $\bbb$. The situation when $\aaa=\Lo$ is usually harder to visualize so most figures have been drawn to depict that case. If $\tau = \emptyset$ we pick the labels $\aaa$ and $\bbb$ arbitrarily.

If  $S=\pi^{-1}(s)$, $s \geq 0$ is a level sphere we call $S$ {\em an alternating sphere for $D^*$} and $s$ an
\textit{alternating level for $D^*$} if the critical point of $L$
just above $S$ and the critical point of $L$ just below $S$ are on
different sides of the c-disk $D^*$. We will also consider $P$ and the lowest level sphere above $D^*$ to be alternating. We will call two alternating
levels $s<s'$ {\em adjacent} if there is no alternating level
between them (but there may be non-alternating levels). Note that if $s'<s$ are two adjacent alternating levels, then one of
$\alpha \cap \pi^{-1}[s',s]$ or $\beta \cap \pi^{-1}[s',s]$ is a
product. 

 If $D$ is a compressing disk for $P$ we can always obtain a cut-disk for $P$ by simply isotoping an endpoint of $L^+$ across $D$. We do not want to consider such cut-disks so
 if $D^c$ is a cut-disk for $P$ and one of $\tau \cap \alpha$ or $\tau \cap \beta$ is parallel to $P\cup D^c$ in the complement of the link then $D^c$ will be called a {\em fake cut-disk}. 
If $D^c$ is a fake cut-disk, let $E$ be the disk of parallelism between $\tau \cap \alpha$ or $\tau \cap \beta$ and $P\cup D^c$. Then the boundary of a regular neighborhood of $D^c \cup E$ contains a compressing disk for $P$, we will call this disk the {\em associated compressing disk}. 
 
We will restrict our attention to tangles such that no component of the tangle is parallel to $S^2\times\{-1\}$ or $S^2\times\{1\}$ in the tangle complement. We will call such tangles {\em proper}. Most results in this paper are stated for proper tangles although many of them can be modified to hold also for general tangles but the statements become cumbersome. Note that if $L \subset S^2 \times [-1,1]$ is a tangle and $S^2 \times \{-1\}$ and $S^2 \times \{1\}$ are incompressible in the tangle complement, then the tangle is proper so restricting our attention to these tangles is natural. When $L$ is a proper tangle we will consider $S^2\times\{-1\}$ and $S^2\times\{1\}$ to be thin level spheres.

\section{Isotoping $D$ to be vertical} \label{sec:makingDvertical}

Let $L$ be a link or a tangle as described above. It will be useful to move parts of $L \subset S^{2} \times [-1,1]$
vertically, that is without changing the projection of $L$ to
$S^{2}$, but altering only the height function $\pi$ on those parts.
Suppose, for example, $a < b$ are regular values for $\pi|_L$.  Take
$\epsilon > 0$ so small that there are no critical values of $\pi|_L$
in either of the intervals $[a, a+\epsilon]$ or $[b, b +
\epsilon]$.  Let $h: [a, b + \epsilon] \rightarrow [a, b +
\epsilon]$ be the homeomorphism that consists of the union of the
linear homeomorphisms $[a, a+ \epsilon] \rightarrow [a, b]$ and
$[a + \epsilon, b + \epsilon] \rightarrow [b, b + \epsilon].$

\begin{defin} Let $\sigma$ be a collection of components of $L \cap
(S^{2} \times [a, b + \epsilon])$.  The push-up of $\sigma$ past
$S^{2} \times \{ b \}$ is the image of $\sigma$ under the
homeomorphism $H: S^{2} \times [a, b + \epsilon] \rightarrow S^{2}
\times [a, b + \epsilon]: (x, t) \mapsto (x, h(t))$. (Figure
\ref{fig:pushup})
\end{defin}

\begin{figure}
\begin{center} \includegraphics[scale=.5]{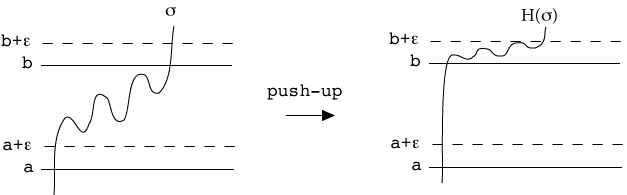}
\end{center}
\caption{} \label{fig:pushup} \end{figure}

Notice that all critical points of $H(\sigma)$ lie in $S^{2} \times
[b, b+\epsilon]$.  Since there is a linear isotopy from
$h$ to the identity, $\sigma$ is properly isotopic to $H(\sigma)$ in
$S^{2} \times [a, b+\epsilon]$.  This isotopy from $\sigma$ to
$H(\sigma)$ is called {\em pushing the critical points of $\sigma$
above the sphere $\pi^{-1}(b)$}.  There is an obvious symmetric
isotopy that pushes the critical points of $\sigma$ below the
sphere $\pi^{-1}(a)$. The following well known fact makes explicit the effect of these isotopies on the width of $L$.

\begin {rmk}\label{rmk:effects}

Let $L$ be a link or a tangle. If a maximum of $L$ is isotoped to lie above (resp. below) a minimum of $L$, the width of the tangle is increased (resp. decreased) by 4. Moving a maximum 
past a maximum or a minimum past a minimum has no effect on the width.

\end{rmk}

These isotopies of $\sigma$ only make sense as
isotopies of $L$ if they do not move $\sigma$ across any
other part of $L$.  In \cite{Tom}, using the argument of \cite{Wu}, we showed that if $L$ is a link and $D$ is a compressing disk for some level sphere for $L$, there is an isotopy that doesn't change the width of $L$ and allows us to assume the disk is vertical. This result made it possible to push up or down $L^+ \cap B^{in}$ in the complement of $L^+ \cap B^{out}$ and vice-versa. The following lemma extends this result to compressing disks for tangles. The proof is essentially the same so we include only  brief sketch for completeness.

\begin{lemma} \label{lem:diskisvertical}
     Let $L$ be a link or a tangle embedded in $S^2 \times [-1,1]$.  Let $P=\pi^{-1}(0)$ be a level 
sphere for $L$ and let $D$ be a compressing disk for $P$ lying above it. Then there is an isotopy of $L$ which restricts to a horizontal isotopy on $S^2 \times [-1,0]$ and on $S^2 \times [a,1]$ and which is level preserving on $\alpha$ and $\beta$ such that after the isotopy we may assume that $D$ is vertical. 
 \end{lemma}    

\begin{proof} 
 By choosing a fiber in $S^2 \times [-1,1]$ that is disjoint from both $L$ and $D$ we can consider $L$ and $D$ to be embedded in $D^2 \times [-1,1]$. We place the standard Cartesian coordinate system 
    on the ball $D^2 \times [-1,1]$ and assume that $\bdd D$ is a circle centered at the origin. Then $P$ lies in the $xy$-plane and the compressing disk $D$
    lies in the upper half space.    
    
     We will build the desired isotopy in 
    several steps.     
    First perform a horizontal isotopy of the upper-half space which is the identity near the $xy$-plane, at each point decreasing the radial 
    distance between $\bdd B^{in}$ and the $z$-axis until $B^{in}$ is 
    entirely contained in $\bdd D \times I$. Next select $\epsilon$ 
    so small that in an $\epsilon$-neighborhood of the $xy$-plane
    $\Lo$ has a product structure. Perform an isotopy 
    $f$ that is the identity on $B^{out}$ and at the end of the isotopy $B^{in}$ is contained in a $\epsilon/2$-neighborhood of the $xy$-plane. 
    This isotopy is not level 
    preserving on $\Li$. Let $E$ be the disk which is the boundary of the ball $\bdd D \times (0, \epsilon]$. It is clear that $E$ is isotopic to $D$ as it is contained in the ball bounded by $D, f(D)$ and $P$.

    Next perform a horizontal isotopy $g$ which is the identity in the $\epsilon$-neighborhood of the $xy$-plane (in particular it is the identity on $E$), and on $B^{in}$ and which
     increases the radial distance between $\bdd B^{out}$ and the positive 
    $z$-axis until $B^{out}$ is entirely contained outside $\bdd D \times 
    I$. Finally perform $f^{-1}$ on $B^{in}$ which restores all critical 
    points of $\alpha$ to their original vertical position.  Note that $f^{-1}$ acts as a vertical isotopy on $E$. The 
    composition of these isotopies gives an isotopy which is level 
    preserving on $\Lo$ and $\Li$. The image of $E$ under $f^{-1}$ is a vertical compressing disk isotopic to $D$. \end{proof}

  \section{Piping one tangle through another tangle}
  
  The push-up isotopy described in the previous section is level preserving when restricted to $B^{in}$ or $B^{out}$. It is however possible for such an isotopy to introduce new critical points in the connecting strand as in Figure \ref{fig:BadPush}.

  \begin{figure}
\begin{center} \includegraphics[scale=.35]{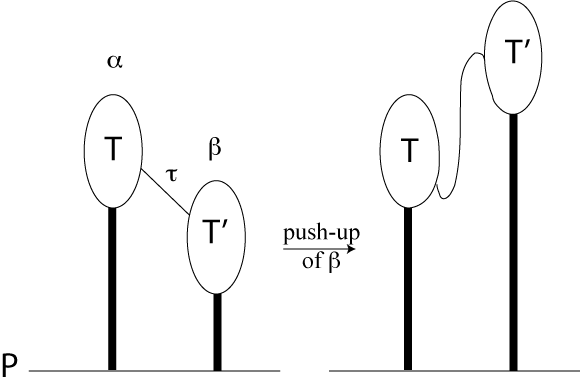}
\end{center}
\caption{An ellipse will always represent a tangle.} \label{fig:BadPush} \end{figure} 

   \begin{rmk} \label{rmk:nonewcritical}
Suppose $D^*$ is a vertical cut-disk. Then an isotopy of $L$ pushing critical points of $\alpha$ up or pushing critical points of $\beta$ down does not introduce any new critical points for $L$. 
    \end{rmk}

 The above remark assures us that we can isotope critical points of $\alpha$ up without introducing new critical points but we will often need to isotope critical points of $\beta$ up. To push up a submanifold of $\beta$ without introducing new critical points we will need a more complicated isotopy which we now describe.

  Suppose $L \subset S^2 \times [-1,1]$ is a link or a tangle and suppose $P =\pi^{-1}(0)$ is a thin level sphere with a vertical c-disk $D^*$. We will continue to denote by $\Lo$ and $\Li$ the components of $L_+$ that lie in the outside and inside of $D^*$ respectively. Let $\tau \cap D^*=p\in \pi^{-1}(c)$ and let $c \pm \epsilon$ be the heights of the endpoints of $\tau'=D^*_\delta \cap \tau$. Consider the first maximum, if there is such, we encounter along $\tau \cap \aaa$ starting from $p$, we will refer to this maximum as {\em the first maximum of $\tau \cap \aaa$}. Similarly we can define the first minimum of $\aaa \cap \tau$ and the first minimum and the first maximum of $\bbb \cap \tau$. Let $R=\pi^{-1}(r)$ be the lowest thin sphere for $\alpha$ above the first maximum of $\alpha \cap \tau$, if there is such. If there is no such maximum let $R=S^2 \times \{1\}$.
If there is a maximum in $\alpha \cap \tau$, by rearranging the maxima in $\alpha$ we may isotope the first such maximum to be the highest maximum of $\alpha$ below $R$. As during this rearrangement maxima are isotoped only above other maxima, the width of $L$ is not changed. We will make the assumption that we have performed this isotopy throughout the paper. 

\begin{figure}
\begin{center} \includegraphics[scale=.5]{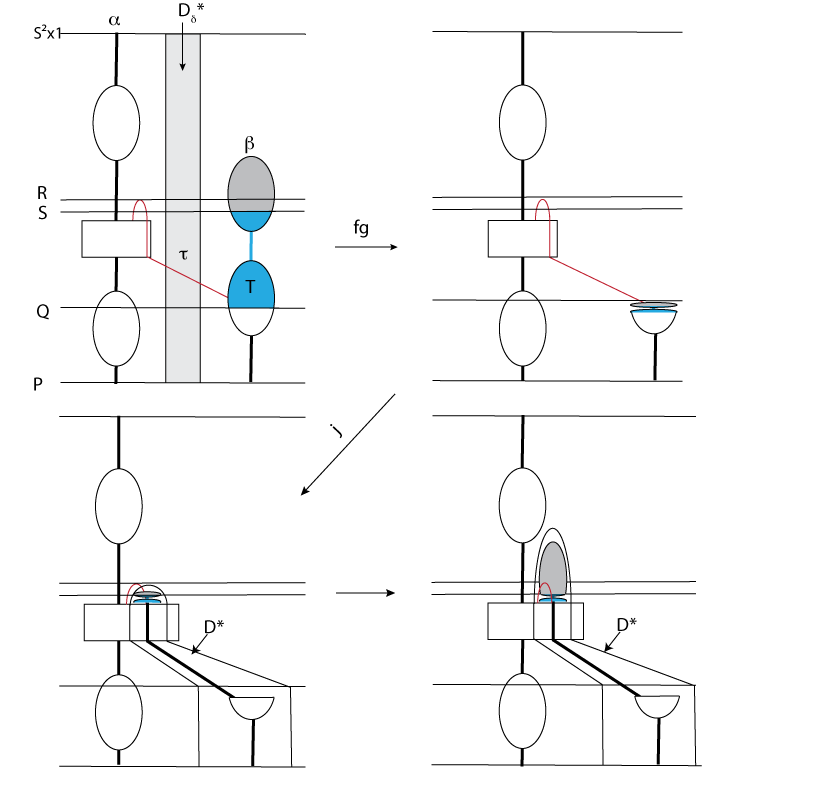}
\end{center}
\caption{An ellipse will always represent any tangle that is a subset of $L$.} \label{fig:pipe} \end{figure}

\begin{prop} \label{prop:piping}
Let $c+\epsilon \leq s< r$ be a regular value of $\pi$ and let $S=\pi^{-1}(s)$. Let $0\leq q \leq c-\epsilon$ be a regular value for $\pi$, $\pi^{-1}(q)=Q$, and let $T$ be the tangle $\pi^{-1}[q,s] \cap \bbb$. Then there is an isotopy of $L$ which preserves the fact that $D^*$ is vertical after which all critical points of $T$ are isotoped to lie in a neighborhood of $S$.
\end{prop}
\begin{proof}

{\bf Case 1: $\alpha=\Lo$.} Consider the following compositions of isotopies, see Figure \ref{fig:pipe}:  let $g$ be the isotopy that pushes down the critical points of $T$ past $\pi^{-1}(q-\delta)$ for some small $\delta$. Then $g(\beta)$ has no critical points in the interval $(q-\delta,s)$ and all critical points of $T$ now lie in the interval $(q-2\delta, q-\delta)$, say. Let $f$ be the isotopy that pushes down $g(\beta) \cap \pi^{-1}[q, 1]$ past $\pi^{-1}(q)$. Then $fg(\beta)$ is disjoint from $\pi^{-1}(t)$ for all $t \geq q$, $fg(\beta) =\beta$ on the interval $(0, q-2\delta)$ and all critical points of $\beta$ that were originally above $S$ now lie in the interval $(q-\delta/2, q-\delta)$ say. Let $j$ be the isotopy that decreases the length of $\tau$ while simultaneously increasing the lengths of the arcs of $\beta$ lying in a small neighborhood of $\pi^{-1}(q-2\delta)$ until all critical points of $T$ lie directly below $S$ and all critical points of $\beta$ that were originally above $S$ again lie above it. Then $hfg(\beta)=\beta$ on $(0, q-2\delta)$ and $hfg(\beta)$ is a product on the interval $(q-2\delta, s-2\delta)$. Next apply $f^{-1}$ to restore the heights of all critical points of $\beta$ above $S$ to their original height. Finally it is clear that there is a horizontal isotopy $k$ that restores $D^*$ to its vertical position. 
 
{\bf Case 2: $\aaa=\Li$.}
\begin{figure}
\begin{center} \includegraphics[scale=.5]{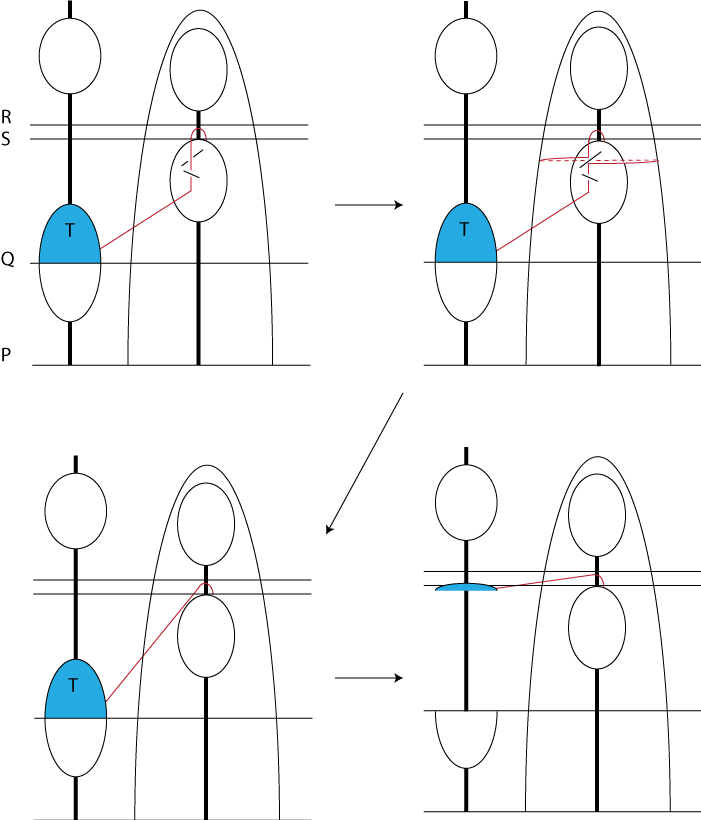}
\end{center}
\caption{} \label{fig:gettingout} \end{figure}
In this case $\tau \cap \aaa$ must have a maximum which, as before, we can assume to lie directly below the thin sphere $R$. Let $q$ and $q'$ be the two points of the intersection $S \cap \tau$ chosen so that $q'$ is between $p$ and $q$. Let $\sigma$ be the component of $\tau-\{p \cup q\}$ containing $q'$ and let $\sigma'$ be the component of $\sigma-q'$ containing $p$ as one endpoint. Consider a projection of $\alpha$  into the $yz$-plane. See Figure \ref{fig:gettingout}. Suppose there are any crossings where $\sigma'$ is the overstrand. Let $u$ be the point of $\sigma'$ that lies over some subarc of $L$ and is closest to $q'$. Isotope a small neighborhood $\gamma$ of $u$ in $\sigma'$ to lie on $D^c$. The arc $\gamma$ together with an arc $\gamma' \subset D^c$ cobound a circle with a single minimum and a single maximum coinciding with the endpoints of $\gamma$ which is the boundary of a subdisk $D'$ of $D^c$. This disk gives an isotopy between $\gamma$ and $\gamma'$. The isotopy is level preserving and decreases the number of overcrossings of $\sigma$. After finitely many iterations we may assume that all crossings of $\sigma$ are undercrossings. In particular we can isotope $\sigma'$ to lie outside $D^c$. 

Let $f'$ be the isotopy that pushes down $T$ until all critical points of $T$ lie in a small neighborhood of $Q$ and let $j'$ be the isotopy of $f'(L)$ that decreases the length of $\sigma'$ while simultaneously increasing the lengths of the arcs of $\beta$ lying in a small neighborhood of $\pi^{-1}(q)$ until all critical points of $T$ lie in a neighborhood of $S$. 

\end{proof} 
\begin{defin} 
The isotopies described in Proposition \ref{prop:piping} will be called {\em piping the tangle $T$ to $S$ along $\tau$}. 
\end{defin}
\begin{rmk} \label{rmk:piperesult} The result of piping a tangle $T \subset \beta$  to $S$ along $\tau$ is that all critical points of $\beta$ that lie in $T$ are isotoped to lie in a neighborhood of $S$ without introducing any new critical points. \end{rmk}

\section{Organizing critical points into braid boxes}

 Let $\sigma$ be a 1-manifold embedded in a ball $B$ containing some critical points 
  with respect to some height function $\pi:B \rightarrow [-1,1]$. We will not consider boundary points of $\sigma$ to be critical points. Recall that a level disk $\pi^{-1}(t)$ is called a {\em thick disk for $\sigma$} if the critical point of $\sigma$ directly above it is a maximum and the critical point directly below it is a minimum. If the positions of the minimum and the maximum are reversed, the disk is called a {\em thin disk for $\sigma$}.  We will also consider $\pi^{-1}(1)$ (resp. $\pi^{-1}(-1)$)  to be a thin disk if the highest (resp. lowest) critical point of $\sigma$ is a maximum (resp. minimum). The critical points of $\sigma$ can be grouped 
  as follows:
   Let $t_1,\ldots,t_k$; $t_i <t_{i+1}$, be a maximal collection of 
   thick disks for $\sigma$ such that $\sigma$ has some critical points between $t_i$ and $t_{i+1}$. Let 
  $t_i^+ > t_i$ be such that $\pi^{-1}(t_i^+)$ is the lowest thin disk for $\sigma$ that is above $\pi^{-1}(t_i)$ and $\pi^{-1}(t_i^+)$ lies directly above  the highest maximum below it, i.e., if $\sigma$ is a submanifold of some 1-manifold $L$, then the highest critical point  of $L$ below  $\pi^{-1}(t_i^+)$ belongs to $\sigma$. Let  $t_i^-< t_i$ be such that $\pi^{-1}(t_i^-)$ is the highest thin disk for $\sigma$ that is below $\pi^{-1}(t_i)$ and $\pi^{-1}(t_i^-)$ lies directly below the lowest minimum of $\sigma$ above it. It is possible that either of these thin disks does not exist. If they do $\pi^{-1}(t_i^+)$ and $\pi^{-1}(t_{i+1}^-)$ are parallel in the 
  complement of $\sigma$ unless $\sigma$ has a boundary point lying between them. We may assume  that $t_i^+ <t_{i+1}^-$.
   The ball
  $\pi^{-1}(t_i^-, t_{i}^+)$ will be called {\em a braid box for $\sigma$}.
 In 
  this region $\sigma$ has a sequence of minima that are below a 
  sequence of maxima.  
  
 If $L \subset S^2 \times [-1,1]$ then we saw there is a natural embedding of $L$ in $D^2 \times [-1,1]$ so the critical points of $L$ can be grouped into braid boxes. If $L$ is a link or a proper tangle, all critical points belong to some braid box.

\section{Vertical c-disks with thin alternating spheres} \label{sec:ifdisjoint}

Suppose $L$ is a link or tangle in thin position, $P$ is a level sphere with vertical c-disk $D^*$. As $P$ is a thin level sphere, the lowest critical point above it is a minimum. In fact the following proposition shows that both $\alpha$ and $\beta$ are proper tanlges
\begin{prop}\label{prop:proper} 
Suppose $L$ is in thin position, $P=p^{-1}(0)$ is a thin level sphere for $L$ and $D^*$ is a vertical c-disk for $P$. Then each of $\alpha$ and $\beta$ is a proper tangle.
\end{prop}

\begin{proof}

Since $P$ is a thin level sphere, the lowest point of $L$ above it is a minimum. Suppose this minimum is contained in $B^{in}$, the other case is symmetric. If the lowest point of $L$ above $P$ in $B^{out}$ is a maximum, this maximum has to lie above some minima in $B^{in}$. A push-down of this maximum either results in pushing a maximum below a minimum decreasing the width of $L$ be $4$, or, if the maximum and the minimum both lie on the connecting strand, the push down results in an elimination of two critical points.  The argument that the highest critical point of each of $\alpha$ and $\beta$ is a maximum is similar and uses the fact that if $L$ is a tangle, it is a proper tangle.
\end{proof}

The above proposition shows and all critical points of $\Li$ and $\Lo$ can be organized into braid boxes. In this section we want to consider the case when no level sphere intersects braid boxes for both $\Li$ and $\Lo$. This assumption has the following nice implication to alternating spheres.

\begin{prop} \label{prop:altimpliesthin}
Suppose $D^*$ is a c-disk (not necessarily vertical) for a thin level sphere $P$ such that each of $\Li$ and $\Lo$ is a proper tangle and no level sphere intersects braid boxes for both $\Li$ and $\Lo$. Then any alternating sphere $S$ for $D^*$ is a thin sphere for $L$.

\end{prop}

\begin{proof}
 Suppose $S$ is an alternating sphere that is not thin, say there are maxima below and above it. Then the level sphere directly below it intersects braid boxes for both $\Li$ and $\Lo$ contradicting the hypothesis. If both critical points are minima, the same result holds for the sphere directly above $S$.

\end{proof}

Suppose $D^*$ is a vertical c-disk for $P$ and all alternating spheres for $D^*$ are thin. Let $S_i=\pi^{-1}(s_i)$, $i=0,..,n$ be the alternating spheres for $D^*$
indexed from the top so that $S_0$ is the highest alternating sphere (which must be the lowest level sphere above $L^{in}$) and $S_n=P$. Recall that $\tau'=D^*_\delta \cap L$.  After possibly a horizontal isotopy of $\tau$ which does not affect the assumption that $D^*$ is vertical we may assume that the endpoints of $\tau'$ lie in the boundaries of braid boxes for $\Li$ and $\Lo$, in particular $\tau'$ intersects at least one alternating sphere. As $\tau'$ is descending from $\alpha$ to $\beta$, if it intersects exactly one alternating sphere the minimum above the sphere, if there is such, must belong to $\alpha$ and the maximum below it must belong to $\beta$. Suppose that $\tau'$ intersects multiple alternating thin spheres, say $S_m, S_{m+1},...S_{m+j}$. Let $S_r \in \{S_m, S_{m+1},...S_{m+j}\}$ be an alternating level sphere so that the minimum above it belongs to $\alpha$ and the maximum below it belongs to $\beta$, such a sphere exists as $j\geq 1$. Pick $\epsilon$ so small that $L$ has no critical points in a $2\epsilon$ neighborhood of $S_r$. Now we can perform a horizontal isotopy of $\tau'$ pushing the part of $\tau'$ that lies above the $2\epsilon$-neighborhood of $S_r$ to lie on the same side of $D^*$ as $\aaa$ and the part below to lie in the side of $D^*$ containing $\bbb$. The result is that after this isotopy the modified $\tau'$ intersects a unique alternating sphere and $D^*$ is still vertical. We will assume that this isotopy has been performed for the reminder of the paper and $S_r$ is the unique alternating sphere that $\tau'$ intersects. As $\tau$ is descending from $\aaa$ to $\bbb$ the critical point of $L$ directly above $S_r$ belongs to $\alpha$ and the one directly below $S_r$ belongs to $\beta$.
Let $M_{\alpha_i}$ and $m_{\alpha_i}$ be respectively the number of maxima and minima
 of $\alpha$ between $S_i$ and $S_{i-1}$.  Define $M_{\beta_i}$ and $m_{\beta_i}$ similarly. As the $S_i$'s are adjacent alternating spheres, exactly one of $M_{\alpha_i}$ and $M_{\beta_i}$ is non-zero for each $i$ and, as $\alpha$ and $\beta$ are proper, $M_{\alpha_i} \neq 0$ if and only if $m_{\alpha_i} \neq 0$, and $M_{\beta_i}\neq 0$ if and only if $m_{\beta_i} \neq 0$.
Note that $|S_{i-1} \cap L| < |S_{i} \cap L|$ if and only if $M_{\alpha_i}+M_{\beta_i}> m_{\alpha_i}+m_{\beta_i}$. 

 \medskip
{\em Suppose $L$ is a tangle in thin position, $P$ is a thin level sphere with a vertical c-disk $D^*$. With the above notation, the following hold.}
\medskip

{\bf Fact 1:} For all $i$ if $M_{\beta_{i}}> m_{\beta_{i}}$ then $M_{\alpha_{i+1}}> m_{\alpha_{i+1}}$ and if $M_{\beta_{i}}= m_{\beta_{i}} \neq 0$ then $M_{\alpha_{i+1}}\geq m_{\alpha_{i+1}}$.
\begin{proof}
Perform an isotopy of $\beta$ pushing all critical points of $\beta$ between $S_{i+1}$ and $S_{i-1}$ to just below $S_{i+1}$. By Remark \ref{rmk:nonewcritical} no new critical points are introduced. The result of the isotopy is that $M_{\beta_i}$ maxima are isotoped down past $m_{\alpha_{i+1}}$ minima and  $m_{\beta_i}$ minima are isotoped down past $M_{\alpha_{i+1}}$ maxima. Thus the width of $L$ is changed by $4(m_{\beta_{i}}M_{\alpha_{i+1}}-M_{\beta_{i}}m_{\alpha_{i+1}})$. As $L$ is thin, the change in the width must be non-negative. Therefore if $M_{\beta_i}>m_{\beta_i}$ then  $M_{\alpha_{i+1}}>m_{\alpha_{i+1}}$  and if  $M_{\beta_i}=m_{\beta_i}\neq 0$ then  $M_{\alpha_{i+1}}\geq m_{\alpha_{i+1}}$
 \end{proof}
 
 {\bf Fact 2:}  Suppose $i \neq r$. If $M_{\alpha_{i}}> m_{\alpha_{i}}$ then $M_{\beta_{i+1}}> m_{\beta_{i+1}}$ and if $M_{\alpha_{i}}= m_{\alpha_{i}} \neq 0$ then $M_{\beta_{i+1}}\geq m_{\beta_{i+1}}$.
\begin{proof}
The proof of this fact is analogous to the proof of Fact 1 but switching the roles of $\alpha$ and $\beta$. Note that the isotopy does not introduce new critical points as $\tau'$ does not intersect the region affected by the isotopy so the relative position of its endpoints does not change. 
 \end{proof}

{\bf Fact 3:} Suppose $r \neq 0$ and the first maximum of $\tau \cap \alpha$ does not lie between $S_{r}$ and $S_{r-1}$. If $M_{\alpha_{r}}> m_{\alpha_{r}}$ then $M_{\beta_{r+1}}> m_{\beta_{r+1}}$ and if $M_{\alpha_{r}}= m_{\alpha_{r}} \neq 0$ then $M_{\beta_{r+1}}\geq m_{\beta_{r+1}}$.
\begin{proof}
The argument here is still analogous to the proof of Fact 1 but to avoid creating new critical points it is necessary to pipe the tangle $T=\beta \cap \pi^{-1}(s_{r+1}, s_r)$ along $\tau$ to $S_{r-1}$.
\end{proof}

{\bf Fact 4:} Suppose $r \neq 0$ and the first maximum of $\tau \cap \alpha$ lies between $S_{r}$ and $S_{r-1}$. If $M_{\alpha_{r}}> m_{\alpha_{r}}$ then $M_{\beta_{r+1}}\geq m_{\beta_{r+1}}$.

\begin{figure}
\begin{center} \includegraphics[scale=.35]{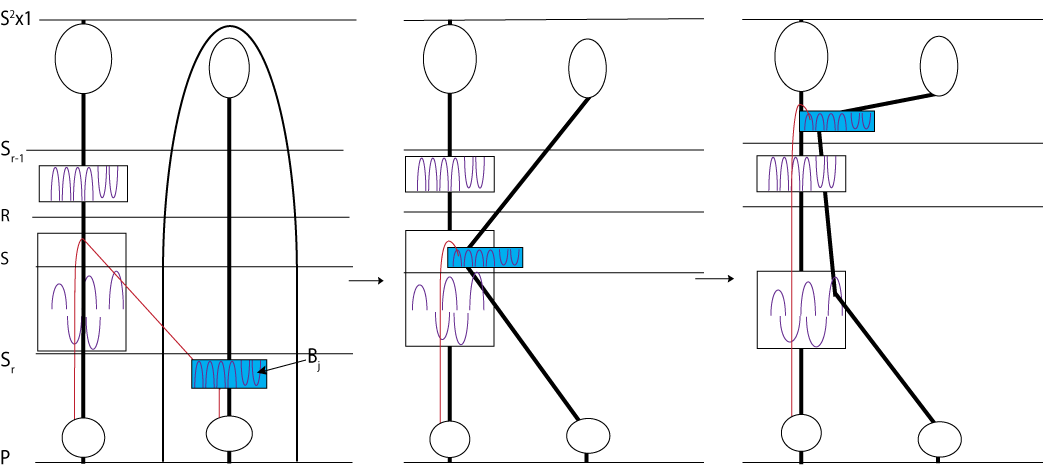}
\end{center}
\caption{A rectangle will always represent a braid box.} \label{fig:Fact4} \end{figure} 

\begin{proof}

Let $R$ be the lowest thin sphere above the first maximum of $\tau\cap \alpha$ and $S$ the level sphere just below it, see Figure \ref{fig:Fact4}. Recall that we have performed an isotopy of $L$ allowing us to assume that the critical point of $\alpha$ between $S$ and $R$ is the first maximum of $\tau \cap \alpha$. It may be that $R=S_{r-1}$ but that is not necessarily the case as $L$ might have (non-alternating) thin spheres between $S_{r}$ and $S_{r-1}$. 
Let $m_{\alpha_r}^-> 0$ be the number of minima of $\alpha$ lying between $S_{r}$ and $R$ and let $m_{\alpha_r}^+ \geq 0$ be the number of minima of $\alpha$ lying between $R$ and $S_{r-1}$. Define $M_{\alpha_r}^+$ and $M_{\alpha_r}^-$ similarly.

We will consider 2 isotopies, see Figure \ref{fig:Fact4}. In the first isotopy, $f$, pipe  
 the tangle $T=\beta \cap \pi^{-1}(s_{r+1}, s_r)$ along $\tau$ to $S$. The result of this isotopy is that $M_{\beta_{r+1}}$ maxima are isotoped above $m_{\alpha_r}^-$ minima and $m_{\beta_{r+1}}$ minima are isotoped above $M_{\alpha_r}^--1$ maxima. Thus the width of $L$ is changed by $4(m_{\alpha_{r}}^-M_{\beta_{r+1}}-(M_{\alpha_{r}}^- -1)m_{\beta_{r+1}})$.  Let $T'$ be the tangle $f(T)$ together with the first maximum of $\tau\cap \alpha$. This tangle has $m_{\beta_r}$ minima and $M_{\beta_r}+1$ maxima. 
Compose the isotopy $f$ with an isotopy $g$ pushing $T'$ up to $S_{r-1}$. The result of the composition of these  two isotopies changes the width of $L$ by 
 $$4(m_{\alpha_{r}}^-M_{\beta_{r+1}}-(M_{\alpha_{r}}^--1)m_{\beta_{r+1}}) + 4(m_{\alpha_{r}}^+(M_{\beta_{r+1}}+1)-M_{\alpha_{r}}^+m_{\beta_{r+1}})=$$
 $$=4(M_{\beta_{r+1}}m_{\alpha_r}+m_{\alpha_{r}}^+-m_{\beta_{r+1}}(M_{\alpha_r}-1)).$$ 

 As $L$ is thin we have the following sequence of inequalities:
 
 $$0 \leq M_{\beta_{r+1}}m_{\alpha_r}+m_{\alpha_{r}}^+-m_{\beta_{r+1}}(M_{\alpha_r}-1)<$$ $$M_{\beta_{r+1}}m_{\alpha_r}+m_{\alpha_{r}}-m_{\beta_{r+1}}(M_{\alpha_r}-1)=(M_{\beta_{r+1}}+1)m_{\alpha_r}-m_{\beta_{r+1}}(M_{\alpha_r}-1).$$
 
\medskip
 So, in summary, $0<(M_{\beta_{r+1}}+1)m_{\alpha_r}-m_{\beta_{r+1}}(M_{\alpha_r}-1)$. As $M_{\alpha_{r}}> m_{\alpha_{r}}$ it follows that  $(M_{\alpha_{r}}-1)\geq m_{\alpha_{r}}$ so $(M_{\beta_{r+1}}+1)> m_{\beta_{r+1}}$ and thus $M_{\beta_{r+1}}\geq m_{\beta_{r+1}}$
 \end{proof}

\section{Alternating spheres for compressing disks}

In this section we will only consider compressing disks and show that if $L$ is in thin position and $D$ is a compressing disk for a thin sphere $P$, then all braid boxes of $\alpha$ and $\beta$ must have disjoint heights and thus all alternating spheres are in fact thin. Recall that by Proposition \ref{prop:proper} all critical points of $\alpha$ and $\beta$ are contained in some braid box. The following lemma is a generalization of \cite[Lemma 3.5]{Tom}.

\begin{lemma} \label{lem:thickaredisjointDISK}
Let $L$ be a link or a proper tangle embedded in $S^2 \times I$, let $P$ be a level 
sphere for $L$ and let $D$ be a compressing disk for $P$. Suppose  $\pi^{-1}[a_{i}^-, a_i^+]$, $i=1,..,n$ and $\pi^{-1}[b_{j}^-, b_j^+]$, $j=1,..,m$ are the collections of braid boxes for $\alpha$ and $\beta$ respectively. Then for any $i$ and $j$, $[a_{i}^-, a_i^+] \cap [b_{j}^-, b_j^+] = \emptyset$.
\end{lemma}

\begin{proof}  By Lemma \ref{lem:diskisvertical} we may assume $D$ is vertical via a level preserving isotopy. Such an isotopy has no effect on the alternating spheres and a level sphere is thin after the isotopy if and only if it was thin before the isotopy.

Suppose for some $i$ and $j$,  $[a_{i}^-, a_i^+] \cap [b_{j}^-, b_j^+] \neq \emptyset$.  This in particular implies that both $\alpha$ and $\beta$ have critical points in the overlapping region. Let  $B^{\alpha}$ is whichever of $B^{in}$ or $B^{out}$ contains $\aaa$ and similarly for $B^{\beta}$ and let $A_i=\pi^{-1}[a_{i}^-, a_i^+] \cap B^{\alpha} $ and $B_j=\pi^{-1}[b_{j}^-, 
  b_j^+]\cap B^{\beta}$. First suppose that $a_i \geq b_j$ (by $a_i= b_j$ we mean that at least one of $\alpha$ or $\beta$ has no critical points between $a_i$ and $b_j$). 
This implies that all maxima of $A_i$ are above all minima of $B_j$. Push all critical points of $\pi^{-1}(a_i^-, b_j^+)\cap \beta$ down to $\pi^{-1}(a_i^-)$, the isotopy is similar to the one depicted in Figure \ref{fig:Lemma441} however in the case under consideration $\tau=\emptyset$. This move slides critical points of $B_j$, at least one of which is a maximum, below critical points of $A_i$. In particular, before the isotopy at least one maximum of $B_j$ was above at least one minimum of $A_i$ and sliding that maximum down decreases the width of $L$ by 4. As no minima are pushed down below maxima the width of $L$ has been decreased, a contradiction.
  
If  $a_i < b_j$ switch the roles of $\alpha$ and $\beta$. As $D$ is a compressing disk, no critical points can be introduced by the isotopy.
\end{proof}

Lemma \ref{lem:thickaredisjointDISK}
 together with the Facts proven in Section \ref{sec:ifdisjoint} allows us to obtain the following useful result.    

  \begin{prop}\label{prop:altspheredecreasingwidthDISK}
 Let $L$ be a link or a proper tangle embedded in $S^2 \times I$, let $P$ be a level 
sphere for $L$ and let $D$ be a compressing disk for $P$. Let $S_i$, $i=0,..,n$ be the alternating spheres for $D$
indexed from the top so that $S_0$ is the highest alternating sphere and $S_n=P$. Let $M_{\alpha_i}$ and $m_{\alpha_i}$ be respectively the number of maxima and minima
 of $\alpha$ between $S_i$ and $S_{i-1}$.  Then for all $i \leq n$,  $M_{\alpha_i}\geq m_{\alpha_i}$ with
equality if and only if $M_{\alpha_i}= m_{\alpha_i}=0$. Similarly for $\beta$. 
\end{prop}

\begin{proof} By Lemma \ref{lem:thickaredisjointDISK} no level sphere intersects braid boxes for both $\alpha$ and $\beta$ so all alternating spheres are thin. Note that $|\beta \cap S_0|= 0$ and $|\beta \cap S_1| \neq 0$ or $|\alpha \cap S_0|= 0$ and $|\alpha \cap S_1| \neq 0$, say the former (the other case is symmetric). Thus $M_{\beta_1}>m_{\beta_1}$ so the proposition holds for $i=1$. For larger values of $i$ we can show that proposition holds by multiple alternating applications of Fact (1) and Fact (2): $M_{\beta_1}>m_{\beta_1}$ implies that $M_{\alpha_2}> m_{\alpha_2}$ which in turn implies that $M_{\beta_3}>m_{\beta_3}$, etc. As only one of $\alpha$ or $\beta$ has critical points between a given pair of adjacent alternating spheres, it also follows that in this case $M_{\beta_{2k}}=m_{\beta_{2k}}=0$ and $M_{\alpha_{2k+1}}= m_{\alpha_{2k+1}}=0$.

\end{proof}    

\begin{cor} \label{cor:Top alternating sphere is thinnest}
Let $L$ be a link or a proper tangle in thin position, let $P$ be a thin level sphere and let $D$ be a compressing disk for $P$. Suppose $S_0,..,S_n$ are the alternating spheres for $D$, then $w(S_0)< w(S_1)<...< w(P)$.
 \end{cor}
 
 \begin{proof}
This follows by Proposition \ref{prop:altspheredecreasingwidthDISK} and the fact that $w(S_i)-w(S_{i+1})=(m_{\alpha_i}+m_{\beta_i})-(M_{\alpha_i}+M_{\beta_i})$ and only one of $\alpha$ or $\beta$ has critical points in the region between $S_{i+1}$ and $S_i$.

\end{proof}

\section{Alternating spheres for vertical c-disks}

	The main result in the previous section is that if $D$ is a compressing disk (which we can assume to be vertical via a level preserving isotopy) and $L$ is in thin position then all alternating spheres are thin. In this section we will show that the corresponding result holds for all vertical c-disks, namely there is an horizontal isotopy $\nu$ so that after the isotopy the braid boxes on opposite sides of the vertical cut-disk have disjoint heights. The proof of this result is considerably harder than the proof of Lemma \ref{lem:thickaredisjointDISK} as we have to be careful not to create additional critical points in the connecting strand. 
		
\begin{lemma} \label{lem:thickaredisjoint}
Let $L$ be a link or a proper tangle embedded in $S^2 \times I$, let $P$ be a level 
sphere for $L$ and let $D^*$ be a vertical c-disk for $P$. If $L$ is not prime we further assume that any decomposing sphere for $L$ intersects $P$. Then there exists a horizontal isotopy $\nu$ which keeps $D^*$ fixed such that if $\pi^{-1}[a_{i}^-, a_i^+]$, $i=1,..,n$ and $\pi^{-1}[b_{j}^-, 
  b_j^+]$, $j=1,..,m$ are the collections of braid boxes for the proper tangles $\alpha $ and $\beta$ respectively, for any $i$ and $j$ $[a_{i}^-, a_i^+] \cap [b_{j}^-, 
  b_j^+] = \emptyset$.
\end{lemma}

\begin{proof} 
Assume that $\nu$ has been chosen amongst all possible horizontal isotopies preserving the fact that $D^*$ is vertical so that the number of non-parallel level spheres that intersect braid boxes for both $\alpha$ and $\beta$ has been minimized. By replacing $L$ with $\nu(L)$ we may assume that $L$ is in thin position, $D^*$ is a vertical c-disk for $P$ and the number of the number of level spheres that intersect braid boxes for both $\alpha$ and $\beta$ has been minimized up to horizontal isotopies of $L$ that fix $D^*$.

  Suppose for some $i$ and $j$,  $[a_{i}^-, a_i^+] \cap [b_{j}^-, b_j^+] \neq \emptyset$.  This in particular implies that both $\alpha$ and $\beta$ have critical points in the overlapping region. Let $A_i=\pi^{-1}[a_{i}^-, a_i^+] \cap B^{\alpha}$ and $B_j=\pi^{-1}[b_{j}^-, 
  b_j^+]\cap B^{\beta}$ where $B^{\alpha}$ is whichever one of $B^{in}$ or $B^{out}$ contains $\alpha$ and similarly for $B^{\beta}$. 
     
 {\bf Case 1:} $a_i \geq b_j$.  (As before  $a_i = b_j$ means that at least one of $\alpha$ or $\beta$ has no critical points between $a_i$ and $b_j$).
 
 The hypothesis of this case implies that all maxima of $A_i$ are above all minima of $B_j$. Push all critical points of $\pi^{-1}(a_i^-, b_j^+)\cap \beta$ down to $\pi^{-1}(a_i^-)$, see Figure \ref{fig:Lemma441}. This move slides critical points of $B_j$, at least one of which is a maximum, below critical points of $A_i$. In particular, before the isotopy at least one maximum of $B_j$ was above at least one minimum of $A_i$ and sliding that maximum down decreases the width of $L$ by 4. As no minima are pushed down below maxima and by Remark \ref{rmk:nonewcritical} no new critical points have been introduced, it follows that the width of $L$ has been decreased, a contradiction.

\begin{figure}
\begin{center} \includegraphics[scale=.5]{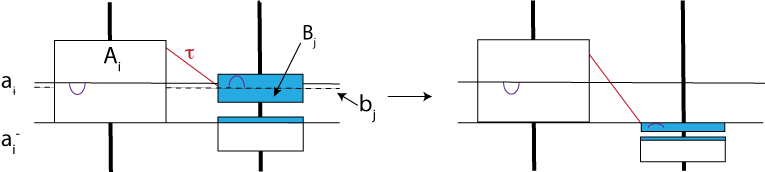}
\end{center}
\caption{} \label{fig:Lemma441} \end{figure} 

 {\bf Case 2:}  $a_i < b_j$ and at least one endpoint of $\tau'$ is not in the region $\pi^{-1}[b_j^-, a_i^+]$.
   
As $a_i \neq b_j$ both $\alpha$ and $\beta$ have critical points between $a_i$ and $b_j$.  In this case the argument is identical to the argument in Case 1 but $\pi^{-1}[b_j^-, a_i^+]\cap \alpha$ is pushed down to $\pi^{-1}(b_j^-)$. As the relative position of the endpoints of $\tau'$ is not affected by this isotopy, no new critical points are introduced. However maxima are isotoped down past at least one minimum so the width of $L$ is decreased, a contradiction. 
  
  {\bf Case 3:} $a_i < b_j$,  both endpoints of $\tau'$ lie in the region $\pi^{-1}[b_j^-, a_i^+]$,  and if the first maximum of $\tau\cap \alpha$ is in $A_i$ some minimum of $B_j$ is below at least two maxima of $A_i$.

  \begin{figure}
\begin{center} \includegraphics[scale=.5]{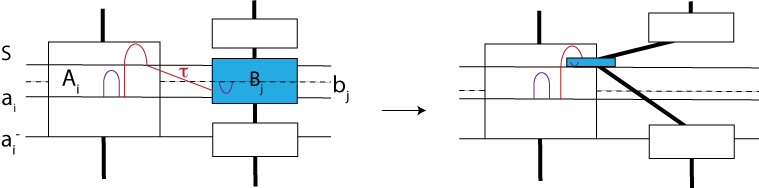}
\end{center}
\caption{} \label{fig:Lemma443} \end{figure}

Recall that we are assuming that if $\tau$ has a maximum in $A_i$ the first such maximum is the highest maximum of $A_i$. Let $S$ be a level sphere directly above all maxima of $A_i$ except the highest maximum if it belongs to $\tau$, (see Figure \ref{fig:Lemma443}), in particular if the first maximum of $\tau$ is not below $\pi^{-1}(a_i^+)$, then $S=\pi^{-1}(a_i^+)$. Pipe the tangle $\pi^{-1}(b_j^-, s) \cap \beta$ along $\tau$ to $S$. The result is that all critical points in $ \pi^{-1}(b_j^-, s) \cap \beta$ have been isotoped to lie above at least one maximum of $A_i$ without introducing any new critical points. By hypothesis $\pi^{-1}(b_j^-, s) \cap \beta$ has at least one minimum before the isotopy, $\pi^{-1}(b_j^-, s) \cap \alpha$ has at least one maximum and if it has any minima they are below any maxima of $B_j$ to begin with. Thus the isotopy decreases the width of $L$, a contradiction.

  {\bf Case 4:} $a_i < b_j$,  both endpoints of $\tau'$ lie in the region $\pi^{-1}[b_j^-, a_i^+]$, the first maximum of $\tau\cap \alpha$ is in $A_i$, it is not the only maximum in $A_i$ and all minima of $B_j$ lie above all other maxima of $A_i$.

  \begin{figure}
\begin{center} \includegraphics[scale=.4]{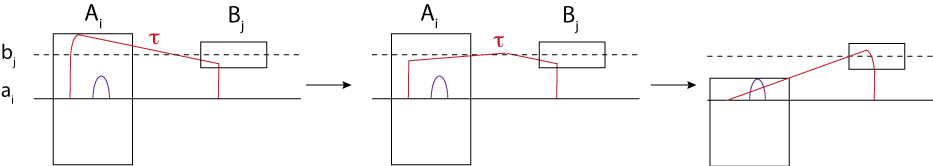}
\end{center}
\caption{} \label{fig:Lemma445} \end{figure} 
   
By the hypothesis of this case $A_i$ has only one critical point above $\pi^{-1}(b_j^-)$. Isotope this maximum horizontally without changing $D^*$ so it becomes a maximum for $\beta$, see Figure \ref{fig:Lemma445}. As $A_i$ has other maxima the effect is that at least one level sphere, $\pi^{-1}(b_j^-)$, is removed from the collection of level spheres which intersect braid boxes for both $\alpha$ and $\beta$ via a horizontal isotopy which fixes $D^*$, a contradiction.

   {\bf Case 5:}  $a_i < b_j$,  both endpoints of $\tau'$ lie in the region $\pi^{-1}[b_j^-, a_i^+]$, the first maximum of $\tau\cap \alpha$ is in $A_i$ and it is the only maximum in $A_i$.

As $a_i<b_j$ and $A_i$ has only one maximum it follows that $b_j >a_i^+$ and so the first critical point of $\beta$ above $\pi^{-1}(a_i^+)$ is a minimum.  

As every decomposing sphere for $L$ intersects $P$, each of $\alpha$ and $\beta$ must intersect $\pi^{-1}(a_i^+)$ at least twice so in fact $\pi^{-1}(a_i^+)$ is compressible with a compressing disk $D \subset D^*$. Then $\pi^{-1}(a_i^+)$ is a thin sphere in $L$:  the critical point below it is the maximum of $\tau\cap \alpha$ and the lowest critical points of both $\alpha$ and $\beta$ above $\pi^{-1}(a_i^+)$ are minima. For $\alpha$ this follows from the definition of $\pi^{-1}(a_i^+)$ and for $\beta$ it follows from the previous paragraph. Let $\alpha'$ and $\beta'$ be the subsets of $\alpha$ and $\beta$ respectively that lie above $\pi^{-1}(a_i^+)$. As both $\alpha'$ and $\beta'$ have a minimum as their lowest critical point and a maximum as their highest one, they are both proper tangles. By Lemma \ref{lem:thickaredisjointDISK} the braid boxes of $\alpha'$ and $\beta'$ have disjoint heights. Let $P'$ be the lowest alternating sphere for $L$ above $\pi^{-1}(a_i^+)$. Let $M_\beta'$ and $m_\beta'$ be respectively the number of maxima and minima of $\beta'$ in the region between $\pi^{-1}(a_i^+)$ and $P'$, and let $M'_{\alpha}$ and $m'_\alpha$ be the number of maxima and minima for $\alpha'$ lying between $\pi^{-1}(a_i^+)$ and $P'$. By the definition of $P'$ either $M_\beta'=m_\beta'=0$ or $M'_{\alpha}=m'_\alpha=0$.

{\bf Subcase 5A:} Suppose that the first critical point of $L$ above $\pi^{-1}(a_i^+)$ belongs to $\beta$ and so it is a minimum as we have already shown (in particular $M'_{\alpha}=m'_\alpha=0$), see Figure \ref{fig:Lemma446A}.

By Proposition \ref{prop:altspheredecreasingwidthDISK}, $M_\beta'>m_\beta'$. Let $m_\alpha$ be the number of minima in $A_i$, $A_i$ has one maximum by the hypothesis of this case. Consider the isotopy pushing the critical points of $\beta'$ contained between $\pi^{-1}(a_i^-)$ and $P'$ down to $\pi^{-1}(a_i^-)$. If $\beta$ has no maxima between $\pi^{-1}(a_i^-)$ and $\pi^{-1}(a_i^+)$, this isotopy changes the width of $L$ by $4(-M_\beta' m_\alpha+m_\beta' )$. As $M_\beta'>m_\beta'$ and $1\leq m_\alpha$, the width is decreased leading to a contradiction. If $\beta$ has maxima between $\pi^{-1}(a_i^-)$ and $\pi^{-1}(a_i^+)$, these maxima are pushed below minima decreasing the width of $L$ even more.

  \begin{figure}
\begin{center} \includegraphics[scale=.3]{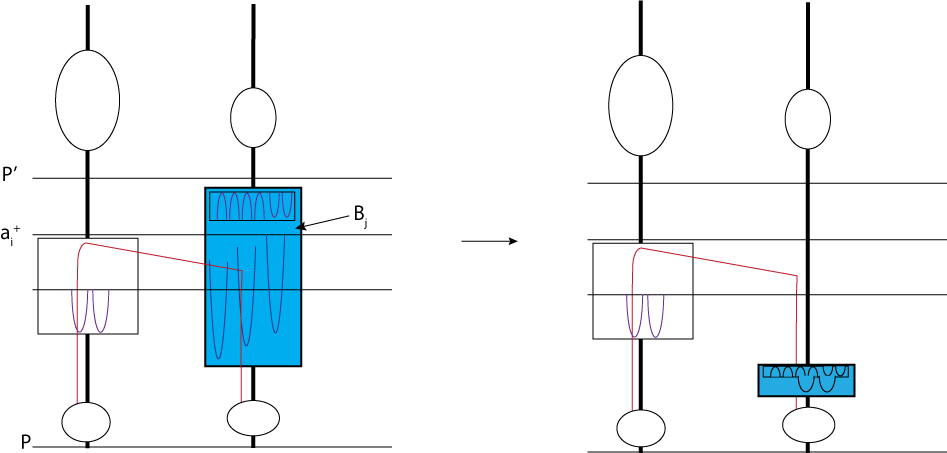}
\end{center}
\caption{} \label{fig:Lemma446A} \end{figure} 

{\bf Subcase 5B:} Finally suppose that the first critical point of $L$ above $\pi^{-1}(a_i^+)$ (necessarily a minimum) belongs to $\alpha$. By Proposition \ref{prop:altspheredecreasingwidthDISK},  $M'_\alpha>m'_\alpha$. Consider the tangle  $T=\pi^{-1}(b_j^-,a_i^+)\cap \beta$. We have already shown that this tangle only contains minima, say $m_\beta$ of them.  Pipe $T$ to $\pi^{-1}(a_i)$ along $\tau$, see Figure \ref{fig:Lemma446B}. This isotopy increases the heights of minima so it cannot increase the width of $L$ and it may decrease it. Now consider the tangle $T'$ consisting of $T$ together with the maximum of $A_i$. Push this tangle up to $P'$. This isotopy changes the width of $L$ by $4(-m_\beta M'_{\alpha}+m'_{\alpha})$. As $M'_\alpha>m'_\alpha$ and $m_\beta\geq 1$, the width is decreased leading to a contradiction.

  \begin{figure}
\begin{center} \includegraphics[scale=.3]{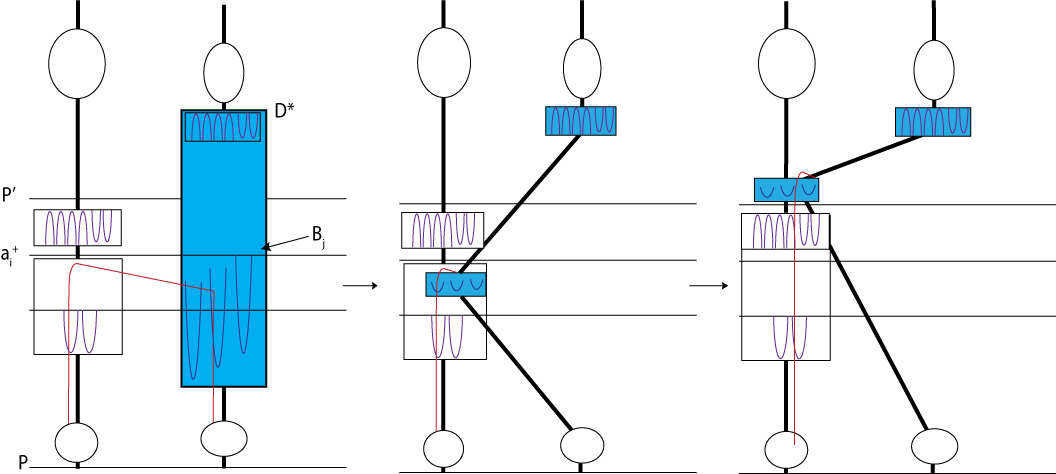}
\end{center}
\caption{} \label{fig:Lemma446B} \end{figure}

 \end{proof}

 \begin{theorem}\label{thm:altspheredecreasingwidth}
Let $L$ be a link or a proper tangle embedded in $S^2 \times I$ that is in thin position, let $P$ be a thin level 
sphere for $L$ and let $D^*$ be a vertical c-disk for $P$. Furthermore assume that $L$ has been isotoped so that all alternating spheres for $L$ are thin.
Let $S_i$, $i=0,..,n$ be the alternating spheres for $D^*$
counting from the top so that $S_0$ is the highest alternating sphere and $S_n=P$. If $D^*$ is a cut-disk let $S_r$ be the unique alternating thin sphere intersected by $\tau'$ and if $D^*$ is a compressing disk let $r=n$. Finally let $M_i$ and $m_i$ be respectively the number of maxima and minima of $L$ between $S_i$ and $S_{i-1}$. Then,
\begin{enumerate}

\item For all $i \leq r$,  $M_{i}> m_{i}$. 

\item If $r=0$ (in particular $D^*$ is a cut-disk) then $M_{{1}} \geq m_{{1}}$ with the equality if and only if $L$ has a decomposing sphere which is disjoint from all alternating spheres.

\item If either

\begin{enumerate}
\item $r \neq 0$ and the first maximum of $\tau \cap \alpha$ does not lie in the region between $S_r$ and $S_{r-1}$, or

 \item $r=0$ and $M_{1}> m_{1}$,
\end{enumerate}
then $M_{i}> m_{i}$ for all $i \leq n$. 

\item If either
\begin{enumerate}
\item $r \neq 0$ and the first maximum of $\tau \cap \alpha$ lies in the region between $S_r$ and $S_{r-1}$, or
\item $r=0$ and $M_{1}= m_{1}$
\end{enumerate}
then for $n \geq i>r$,  $M_{i}\geq m_{i}$.
 However if there is a $j>r$ such that $M_{j}> m_{j}$ then  $M_{i}> m_{i}$ for all $j\leq i\leq n$.

\end{enumerate}

\end{theorem}

\begin{proof}
As $S_i$ and $S_{i-1}$ are adjacent alternating spheres, either $M_i=M_{\alpha_i}$ and $M_{\bbb_i}=0$ or $M_i=M_{\beta_i}$ and $M_{\alpha_i}=0$. Similarly for $m_i$.

{\bf  Conclusion 1:} This follows directly from Proposition \ref{prop:altspheredecreasingwidthDISK}. 

{\bf Conclusion 2: } If $r=0$ then $|\Li  \cap S_r|=1$. As $| \Li  \cap S_{r+1}| \geq 1$ it follows that $M_{r+1} \geq m_{r+1}$ with the equality if and only if $|\Li \cap S_{r+1}|=1$.
In that case $D^*$ and $S_1$ cobound a decomposing sphere for $L$ that is disjoint from all alternating level spheres for $D^*$.

{\bf Conclusion 3: } If $r \neq 0$, for all $i \leq r$ this conclusion is just a restatement of Conclusion (1). From that we know that $M_{{r}}> m_{{r}}$. As the first maximum of $\tau$ does not lie between $S_r$ and $S_{r-1}$, from Fact (3) it follows $M_{{r+1}}> m_{{r+1}}$. If $r=0$, $M_{{r+1}}> m_{{r+1}}$ by hypothesis. Now, as in Proposition \ref{prop:altspheredecreasingwidthDISK}, multiple alternating applications of Fact (1) and Fact (2) give the desired result for all $r+2 \leq i \leq n$.

{\bf Conclusion 4: } If $r \neq 0$, by Conclusion (1) we know that $M_{{r}}> m_{{r}}$. By Fact (4), it follows that $M_{{r+1}}\geq m_{{r+1}}$. If $r=0$, $M_{{r+1}}= m_{{r+1}}$ by hypothesis.  Again alternate applications of Fact (1) and Fact (2) give the desired result (in this case, non-strict inequalities).   However if there is a $j>r$ such that $M_{j}> m_{j}$ or $M_{j}> m_{j}$ then  Fact (1) and Fact (2) give strict inequalities for $j \leq i \leq n$.

\end{proof}

 \begin{cor}\label{cor:thindecreasing}
 
   Suppose $L$ is a link or a proper tangle in thin position embedded in $S^2 \times I$, let $P$ be a thin level 
sphere which intersects any decomposing sphere for $L$, let $D^*$ be either a compressing disk or a vertical cut-disk for $P$ which is not a fake cut-disk. Suppose $L$ has been isotoped, as is always possible, so that all alternating spheres are thin. If $S_0,..,S_n$ are the alternating spheres for $D^*$ in $L$, then $w(S_0)< w(S_1) <...<w(S_r)\leq w(S_{r+1})\leq...\leq w(P)$ and if for some $j> r$ we have $w(S_{j}) < w(S_{j+1})$, then $w(S_i) < w(S_{i+1})$ for all $i>j$.
\end{cor}

\begin{proof}
This follows from Theorem \ref{thm:altspheredecreasingwidth} and the fact that $w(S_{i+1})-w(S_{i})=M_{i}-m_{i}$.
\end{proof}

\begin{cor}\label{cor:not compressible on one side}
Let $L$ be a link or a proper tangle in thin position and suppose $P$ is a thin sphere which intersects any decomposing spheres for $L$. Let $D^*$ be a compressing disk or a vertical cut-disk for $P$ and let $S_0$ be the lowest thin level sphere above $D^*$. Then $w(S_0)<w(P)$.

\end{cor}

\begin{proof}
By possibly replacing $D^*$ with the associated compressing disk, we may assume $D^*$ is not a fake cut-disk. Let $S_0,..,S_n$ be the alternating spheres for $D^*$ in $L$. Then $w(P)<w(S_0)$: this follows by Corollary \ref{cor:thindecreasing} if $r\neq 0$. If $r=0$ Conclusions (2) and (3) of Theorem \ref{thm:altspheredecreasingwidth} establish that $w(S_0)<w(S_1)<...<w(S_n)=w(P)$.

\end{proof}

\begin{cor} \label{cor:incomp} Let $L$ be a link or proper tangle in thin position and suppose $P$ is a thin sphere which intersects any decomposing spheres for $L$ and $P$ has minimum width amongst all thin spheres for $L$. Then $P$ is incompressible and does not have any vertical cut-disks.
\end{cor}

\begin{proof}
Follows immediately from Corollary \ref {cor:not compressible on one side}.
\end{proof}

 \end{document}